\documentclass[11pt]{amsart}
\usepackage{amssymb,latexsym,amsmath,amsthm,enumitem,mathrsfs,geometry}
\usepackage{fancyhdr}
\usepackage{color}
\usepackage{stmaryrd}
\usepackage{mathrsfs}
\usepackage{hyperref}
\hypersetup{
    colorlinks=true,
    linkcolor=red,
    filecolor=magenta,
    urlcolor=cyan,
    citecolor=green
    }
\usepackage{lineno}
\usepackage{diagbox}
\usepackage{array}
\usepackage{tikz}
\usetikzlibrary{arrows}
\usetikzlibrary{positioning}
\usepackage{float}
\usepackage{bbm}
\usepackage{cleveref}
\usepackage{dsfont}
\usepackage{graphicx}
\usepackage{soul}
\usepackage{verbatim}
\usepackage{comment}
\usepackage{caption}

\newtheorem{theorem}{Theorem}
\newtheorem{prop}{Proposition}
\newtheorem{corollary}{Corollary}
\newtheorem*{theorem*}{Theorem}
\newtheorem*{ESH}{Essential Simplicity Hypothesis (ESH)}
\newtheorem*{PCC}{Pair Correlation Conjecture (PCC)}
\newtheorem*{AH_P}{The Alternative Hypothesis for Pairs of Zeros (AH-Pairs)}
\newtheorem*{AH_WD}{The Alternative Hypothesis for Weak Density of Pairs of Zeros (AH-Weak Density)}
\definecolor{pink}{rgb}{1,.2,.6}
\definecolor{orange}{rgb}{0.7,0.3,0}
\definecolor{blue}{rgb}{.2,.6,.75}
\definecolor{green}{rgb}{.4,.7,.4}
\definecolor{purple}{RGB}{127,0,255}

\newcommand{\Ss}{D}

\numberwithin{equation}{section}
\allowdisplaybreaks

\begin{document}
\title[AH Pair Correlation and Essential Simplicity]{Pair Correlation Conjecture for the zeros of the Riemann zeta-function II: The Alternative Hypothesis}

\author[Goldston]{Daniel A. Goldston}
\address{Department of Mathematics and Statistics, San Jose State University}
\email{daniel.goldston@sjsu.edu}

\author[Lee]{Junghun Lee}
\address{Department of Mathematics Education, Chonnam National University}
\email{junghun@jnu.ac.kr}

\author[Schettler]{Jordan Schettler}
\address{Department of Mathematics and Statistics, San Jos\'{e} State University}
\email{jordan.schettler@sjsu.edu}

\author[Suriajaya]{Ade Irma Suriajaya}
\address{Faculty of Mathematics, Kyushu University}
\email{adeirmasuriajaya@math.kyushu-u.ac.jp}

\dedicatory{\lq\lq A theory has only the alternative of being right or wrong. A model has a third possibility: it may be right, but irrelevant." 
- Manfred Eigen}


\begin{abstract} In an earlier paper, we proved that Montgomery's Pair Correlation Conjecture (PCC) for zeros of the Riemann zeta-function can be used to prove without the assumption of the Riemann Hypothesis (RH) that asymptotically 100\% of the zeros are both simple and on the critical line. This is based on a method of Gallagher and Mueller from 1978. We formulate an appropriate form of the Alternative Hypothesis (AH), which determines a different PCC, and, using the same method as above, prove that asymptotically, 100\% of the zeros are both simple and on the critical line. As in our previous paper, we do not assume RH.
\end{abstract}
\date{\today}

\maketitle


\section{Introduction and Statement of Results}
\label{sec1}

In our previous paper \cite{GLSS1}, we consider the implication of Montgomery's pair correlation conjecture for zeros of the Riemann zeta-function $\zeta(s)$ on their simplicity and horizontal distribution. In this paper, we consider a different pair correlation conjecture, which is supposedly a {\it counterhypothesis} to Montgomery's pair correlation conjecture.
We give a brief outline of this terminology.
Zeros of $\zeta(s)$ lie either at negative even integers $s=-2n$ ($n\in \mathbb{N}$) or within the strip $\{s\in \mathbb{C} : 0<{\rm Re}(s)<1 \}$ called the \lq\lq critical strip". The former zeros are called \lq\lq trivial" since their exact location is perfectly determined, and they are all known to be simple. Meanwhile, the latter zeros are only known to satisfy certain properties, including being non-real and symmetric with respect to both the real axis ${\rm Im}(s)=0$ and the \lq\lq critical line" ${\rm Re}(s) = 1/2$.
These zeros in the critical strip are called \lq\lq non-trivial" and are the main object in this paper.
The Riemann Hypothesis (RH) states that all such zeros lie on the critical line and was usually assumed when studying the pair correlation of these zeros.
That is, previous authors have mostly  considered the vertical spacings between two zeros by assuming that these zeros all lie on the critical line.
In recent work \cite{BGST-PC,BGST-CL,GLSS1}, it is shown that in many cases, this assumption is not necessary, that is, we can remove RH in the analysis.

Throughout this paper we do not assume RH; any results we mention that depend on RH are only used as a model to formulate conjectures. We retain the same notation as in \cite{GLSS1} except for a few simplifications.
We denote the non-trivial zeros of $\zeta(s)$ by $\rho=\beta+i\gamma$, thus we have for each $\rho$, that $0<\beta<1$ and $\gamma \in \mathbb{R} \backslash \{ 0 \}$, and there is a reflected zero $\overline{\rho}=\beta-i\gamma$.
Furthermore, for each $\rho$ with $\beta \neq 1/2$, we have additionally the corresponding symmetric zero $1-\overline{\rho} = 1-\beta+i\gamma$ and the reflected symmetric zero $1-\rho = 1-\beta-i\gamma$, so it then suffices to consider only zeros $\rho=\beta+i\gamma$ in the upper half-plane, i.e. $\gamma>0$.

Next, we let $N(T)$ denote the number of zeros $\rho=\beta+i\gamma$ of $\zeta(s)$ with $0<\gamma\le T$, counted with multiplicity. It is known that
\begin{equation}\label{N(T)} N(T) := \sum_{\substack{\rho\\ 0<\gamma \le T}}1 \sim TL,
\qquad \text{where}\quad L:=\frac1{2\pi}\log T, \end{equation}
see \cite[Theorem 9.4]{Titchmarsh2}.
Thus we see that the average vertical spacing between zeros is $1/L$.
By counting zeros with multiplicity we mean that, for example, a triple zero is counted three times in the sum.
Accordingly, we define the pair correlation counting function $N(T,\lambda)$ by
\begin{equation}\label{N(T,U)} N(T,\lambda) := 
\sum_{\substack{\rho,\rho' \\ 0<\gamma,\gamma'\le T \\ 0<(\gamma-\gamma')L \le \lambda}} 1,
\end{equation}
which counts the number of pairs of zeros $\rho=\beta+i\gamma$ with $0<\gamma\le T$, where the height of the second zero is greater than the first by at most $\lambda/L$.
In \eqref{N(T,U)}, we have altered the definition (1.3) in \cite{GLSS1} to avoid using the parameter $U$. This is simply done by replacing $U$ in \cite{GLSS1} with $\lambda/L$.
Montgomery's \cite{Montgomery73} pair correlation conjecture is as follows.
\begin{PCC}\label{PCC} For $\lambda>0$, then
\begin{equation}\label{PCCresult}
N(T,\lambda) = TL \int_0^{\lambda} \left(1 - \left(\frac{\sin\pi\alpha}{\pi \alpha}\right)^2\right) d\alpha + o(TL), \qquad \text{as}\qquad T\to \infty,
\end{equation}
uniformly in each interval $0 < \lambda_0 \le \lambda \le \lambda_1 <\infty$. 
\end{PCC}

Montgomery \cite{Montgomery73} also showed how to apply his approach to obtain that almost all zeros $\rho$ of $\zeta(s)$ are simple.
More precisely, if we let $m_\rho$ denote the multiplicity of the zero $\rho$ and define
\begin{equation}\label{N*} N^*(T) := \sum_{\substack{\rho\\ 0<\gamma\le T}} m_\rho = \sum_{\substack{\rho\ \text{distinct} \\ 0<\gamma \le T}} m_\rho^2, \end{equation}
then the following holds conjecturally:
\begin{equation}\label{SMCeq}
N^*(T) = TL+o(TL), \qquad \text{as}\quad T\to\infty.
\end{equation}
We refer to this as the Simple Multiplicity Conjecture (SMC).
Note that in the last summation in \eqref{N*}, each zero $\rho$ is only counted once, whereas in the first sum $\rho$ is counted with multiplicity, hence a triple zero gets counted nine times in either of the sums in \eqref{N*}.
Gallagher and Mueller \cite{GaMu78} proved under RH that \hyperref[PCC]{\bf PCC} implies SMC by taking their sums to be over only the critical zeros $1/2+i\gamma$.
Their method, however, does not actually depend on RH, as clarified in \cite{GLSS1}.

In addition to \eqref{SMCeq}, \hyperref[PCC]{\bf PCC} also implies the following repulsion between zeros:
\begin{equation}\label{repulsion}
N(T,\lambda_0)\ll \lambda_0TL + o(TL) = o(TL), \qquad \text{if}\quad \lambda_0\to 0 ~\text{ as }~ T\to \infty.
\end{equation}
The two properties \eqref{SMCeq} and \eqref{repulsion} together are referred to as \lq\lq Essential Simplicity" of the zeros $\rho$, see Mueller \cite{Mue83}.
In general situations when RH may not be true, we need to slightly alter this notion of {\it Essential Simplicity}.
This is done in \cite{GLSS1} as follows.
\begin{ESH}\label{ESH}
Let
\begin{equation}\label{N_circledast}
N^\circledast(T) := \sum_{\substack{\rho, \rho'\\ 0<\gamma,\gamma' \le T\\ \gamma = \gamma'}} 1.
\end{equation}
Then we have
\begin{align}
& N^\circledast(T) = TL+o(TL)\quad
\text{as $T\to \infty$, and}
\label{ES1}\tag{\textbf{ES1}} \\
& N(T,\lambda_0) = o(TL)
\quad
\text{if $\lambda_0 \to 0$ as $T \to \infty$.}
\label{ES2}\tag{\textbf{ES2}}
\end{align}
\end{ESH}

We remark that \hyperlink{ESH}{\bf ESH} is equivalent to 
\[\label{ES}\tag{\textbf{ES}}
    N^\circledast(T) + 2 N(T,\lambda_0) = TL+o(TL)
    \quad
    \text{if $\lambda_0 \to 0$ as $T \to \infty$.}
\]
In fact, if we assume \eqref{ES},
then the fact that
\begin{equation}\label{N_cir_star_lower}
N^\circledast(T)\ge N^*(T) \ge N(T) = TL+o(TL),
\end{equation}
implies that $N^\circledast(T) = TL+o(TL)$ as $T\to\infty$, which is \eqref{ES1}, and that forces $N(T,\lambda_0)=o(TL)$ as $T\to\infty$, which is \eqref{ES2}.
In \cite{GLSS1}, we showed that \hyperref[PCC]{\bf PCC} implies \hyperref[ESH]{\bf ESH}, which then implies that asymptotically 100\% of the zeros $\rho$ are simple and on the critical line.

\medskip

In this paper our goal is to replace \hyperref[PCC]{\bf PCC} by an appropriate form of the Alternative Hypothesis (AH) for zeros of $\zeta(s)$. The main reason for doing this is that the Gallagher-Mueller method we use as our main tool does not depend on RH, and therefore we can also remove RH as an assumption in our results on AH.

The Alternative Hypothesis has been studied by many authors \cite{ConRH,FGL14,Bal16,LR20,BGST-AH}. In particular, Baluyot \cite{Bal16} provides an extensive and elaborate exposition on the subject.
AH first arose as a consequence of the possible existence of Landau-Siegel zeros of real Dirichlet $L$-functions, as demonstrated by Heath-Brown \cite{Heath-Brown96}.
To describe this briefly, a sequence of Landau-Siegel zeros forces the existence of infinitely many extremely long intervals $(a,b]$ where zeros $\rho=\beta+i\gamma$ with $a<\gamma\le b$ all satisfy $\beta=1/2$, are simple, and are spaced at nearly integer multiples of half the average spacing.
Regardless of the existence of Landau-Siegel zeros, these hypothetical properties of the zeros of $\zeta(s)$ could still hold. This hypothesis is referred to as the {\it Alternative Hypothesis} (AH).
On the other hand, experimental evidence \cite{MO,Odl87} on the vertical distribution of zeros supports \hyperref[PCC]{\bf PCC} and contradicts AH.

Our starting point is the following formulation of AH from \cite{BGST-AH}.
\begin{AH_P}\label{AH_GLSS}
Let $M$ be any positive real number and
\begin{equation}\label{P(T,M)} \mathcal{P}(T,M) := \left\{ (\gamma,\gamma') : \frac{T}{\log^2\!T} < \gamma,\gamma'\le T,~ \left|(\gamma-\gamma')L\right| \le M\right\}. \end{equation}
Then for every $(\gamma,\gamma') \in \mathcal{P}(T,M)$ there is an integer $k$ such that
\begin{align}\label{AH2k}
(\gamma-\gamma')L = \frac{k}{2} + O\left((|k|+1) R(T)\right), \tag{\textbf{AH0}}
\end{align}
for a positive decreasing function $R(T)$ such that $R(T) \to 0$ as $T\to \infty$.
\end{AH_P}
\noindent
Note that the range of zeros in \eqref{P(T,M)} is $T/\log^2\!T< \gamma \le T$ instead of the usual range $0<\gamma \le T$.
We explain this in Section \ref{sec2} and also show that it does not affect any of our estimates.

We now define for each integer $k$ satisfying $|k/2|\le M$,
\begin{equation} \label{Bk/2}
\mathcal{B}_{k/2}(T) := \left\{(\gamma,\gamma') \in \mathcal{P}(T,M): \frac{k}2 - \frac14 < (\gamma-\gamma')L \le \frac{k}2 + \frac14\right\},
\end{equation}
which, for sufficiently large $T$, is well-defined and partitions $\mathcal{P}(T,M)$ over $|k|\le 2M$.
Note that when $|k|>2M+1/2$, then $\mathcal{B}_{k/2}(T)$ is the empty set.
We count the \lq\lq density" of these pairs closest to $k/2$ with\footnote{Note that for all $k\in\mathbb{Z}$ we have $P_{k/2} = P_{-k/2}$.}
\begin{equation}\label{Pk/2} P_{k/2}(T) := (TL)^{-1}|\mathcal{B}_{k/2}(T)|. \end{equation}

By the unconditional estimate from \cite[Lemma 9]{GM87}, we have that, for $0\le h\le T$,  
\begin{equation}\label{GMbound}
\sum_{\substack{\rho,\rho' \\ 0<\gamma,\gamma'\le T \\ |(\gamma-\gamma')L|\le h}}1 \le \sum_{\substack{\rho,\rho' \\ 0<\gamma\le T \\ |(\gamma-\gamma')L|\le h}}1 \ll (1+h)TL,
\end{equation}
we have immediately on taking $h=M$ that 
\[\sum_{|k|\le 2M}|\mathcal{B}_{k/2}(T)|= |\mathcal{P}(T,M)|\ll MTL.\]
Thus we conclude
\begin{equation} \label{P_k/2sumbound}\sum_{|k|\le 2M} P_{k/2}(T) \ll M, \end{equation}
and the densities are unconditionally bounded for any fixed $M$, and the average of the densities are also bounded when $M\to \infty$. 
Assuming RH, Baluyot et al. \cite[Theorem 1]{BGST-AH} proved a more precise result on these densities.
\begin{theorem*}
Assuming RH and \hyperref[AH_GLSS]{\bf AH-Pairs}, we have as $T\to\infty$,
\begin{equation}\label{thm1a} 1+o(1) \le P_0(T) \le \frac32 - \frac{2}{\pi^2} + o(1) = 1.29735\ldots + o(1), \end{equation}
and for $k\in \mathbb{Z}\setminus\{0\}$,
\begin{equation}\label{thm1b} P_{k/2}(T) \sim \begin{cases}
\displaystyle P_0(T)-\frac12, & \text{if $k\neq 0$ is even,} \\[2mm]
\displaystyle \frac32-\frac{2}{\pi^2k^2}-P_0(T), & \text{if $k$ is odd.}
\end{cases}
\end{equation}
\end{theorem*}
We now define the ``limiting densities" $p_{k/2}$ by
\begin{equation} \label{pk/2}
p_{k/2} := \lim_{T\to\infty} P_{k/2}(T), \qquad \text{provided this limit exists},
\end{equation}
and on RH we see by \eqref{thm1b} all these limiting densities exist provided $p_0$ exists. Our first result in this paper connects \hyperref[AH_GLSS]{\bf AH-Pairs} with \hyperref[ESH]{\bf ESH}.

\begin{theorem}\label{p0thm}
Assuming \hyperref[AH_GLSS]{\bf AH-Pairs}, we have that $p_0=1$ is equivalent to \hyperref[ESH]{\bf ESH}.
\end{theorem}
We also know from \cite[Theorem 2]{GLSS1}, that assuming \eqref{ES1} of \hyperref[ESH]{\bf ESH}, then asymptotically 100\% of the zeros are simple and on the critical line, which immediately gives the following result.
\begin{corollary}\label{cor1} Assuming \hyperref[AH_GLSS]{\bf AH-Pairs}, if $p_0=1$, then asymptotically 100\% of the zeros of $\zeta(s)$ are simple and on the critical line.
\end{corollary}

If $p_0=1$ we obtain from \eqref{thm1b} and RH that
\begin{equation}\label{RHdensity} p_{k/2} \sim \begin{cases}
\displaystyle \frac12, & \text{if $k\neq 0$ is even}, \\[2mm]
\displaystyle \frac12-\frac{2}{\pi^2k^2}, & \text{if $k$ is odd}.
\end{cases}
\end{equation}

Our use of RH in \eqref{thm1b} and \eqref{RHdensity} is required because we have obtained these results by applying Montgomery's theorem for his function $F(\alpha)$ \cite{Montgomery73}. This, together with $p_0=1$, provides enough information to completely determine $p_{k/2}$ for all $k$. 

In this paper we eliminate the use of RH and replace it with the method of Gallagher and Mueller we used in our earlier paper \cite{GLSS1}. The method only obtains asymptotic formulas when we take long averages over the densities from \hyperref[AH_GLSS]{\bf AH-Pairs}. Our main result is as follows. 
\begin{theorem}\label{thm2} Assuming \hyperref[AH_GLSS]{\bf AH-Pairs}, then for any sufficiently large even integer $M$ and $T \to \infty$, we have
\begin{equation}\begin{split} \label{mainthm} 2\sum_{j=1}^{M}& \left(M-j\right)\left(P_{j-1/2}(T) + P_j(T)-\left(1-\frac{2}{\pi^2(2j-1)^2}\right) \right)\\& = \left(\frac32 -P_0(T)\right)M -\sum_{j=1}^M P_{j-1/2}(T) + O\left(\sqrt{\log M}\right) + O\left(M^2\left(R(T)+\frac1{L^2}\right)\right). \end{split}
\end{equation}
\end{theorem}
We now formulate the simplest model for the densities that is consistent with Theorem \ref{thm2}. By \eqref{P_k/2sumbound}, the right-hand side of \eqref{mainthm} is $\ll M$ as $M\to \infty$ and therefore for the left-hand side of the equation to satisfy this bound it will need to be $o(1)$ as $T\to \infty$ for all except a finite number of terms. The simplest assumption is that this sum vanishes up to an error term $o(1)$ for all terms.
\begin{AH_WD}\label{AH_Weak}
For any sufficiently large $T$, we have for each positive integer $j$,
\begin{align}\label{SumtwoPs}
\displaystyle P_{j-1/2}(T) + P_j(T) = 1-\frac{2}{\pi^2(2j-1)^2} + O(R_P(T)),
\tag{\textbf{AH1}}
\end{align}
where $R_P(T)$ is a positive decreasing function with $R_P(T)\to 0$ as $T\to\infty$, and for any large even integer $M$,
\begin{align}\label{SumHalfk}
\displaystyle \sum_{j=1}^M P_{j-1/2}(T) = \frac{M}2 - \frac14 + O\left(\frac1M\right) + O(MR_P(T)).
\tag{\textbf{AH2}}
\end{align}
Here \eqref{SumHalfk} is uniform in $M$ on any interval $0<m_1\le M\le m_2<\infty$.
\end{AH_WD}
Note that \eqref{SumtwoPs} is obtained immediately from \eqref{thm1b}. Our model for \eqref{SumtwoPs} is that we know the densities $P_{j-1/2}(T) + P_j(T)$ on the intervals $(j+1/2, j+5/4]$, and we also take $P_{-j-1/2}(T) + P_{-j}(T) = P_{j-1/2}(T) + P_j(T)$, but make no assumption on $P_0(T)$. Further, unlike in \eqref{AH2k}, the error in \eqref{SumtwoPs} should not depend on $j$.\footnote{The error term in \eqref{AH2k} measures how close the distance between two zeros is to $k/2$ times the average spacing. We would usually expect about $\asymp k/2$ zeros between these two zeros, and thus if the error between each of the consecutive zeros is $O(R(T))$ then on adding we get an error $O(k R(T))$ for the original pair of zeros. The error term in \eqref{SumtwoPs} is for two consecutive densities, and by \eqref{thm1b} these densities do not grow in $k$ assuming RH.} Thus, on adding these densities, we obtain
\begin{align}\label{averageP_{k/2}}
\sum_{k=1}^{2M}P_{k/2} = \sum_{j=1}^M \left(P_{j-1/2}(T) + P_j(T)\right) &= \sum_{j=1}^M \left(1-\frac{2}{\pi^2(2j-1)^2}\right) + O(MR_P(T)) \notag\\
&= M - \frac14 + O\left(\frac1M\right) + O(MR_P(T)), \qquad \text{as}\quad T\to\infty, 
\end{align}
and therefore we have 
\begin{equation}\label{Sumkk} \sum_{j=1}^M P_j(T) = M - \frac14 - \sum_{j=1}^M P_{j-1/2}(T) + O\left(\frac1M\right) + O(MR_P(T)). 
\end{equation}
On applying \eqref{SumHalfk}
 we obtain
\begin{align} \label{Sumk}
\sum_{j=1}^M P_j(T) = \frac{M}2 + O\left(\frac1M\right) + O(MR_P(T)).
\end{align}
Thus, if we wanted, we could have used \eqref{Sumk} for \eqref{SumHalfk}, and together \eqref{SumHalfk} and \eqref{Sumk} guarantee the existence of equal limiting average densities of $1/2$ on half-integers and also on integers. 

If we substitute \eqref{SumtwoPs} into Theorem \ref{thm2} and use \eqref{Sumk} we immediately obtain the following average form of \eqref{thm1b}.

\begin{theorem} \label{thm3} Assuming \hyperref[AH_GLSS]{\bf AH-Pairs} and \eqref{SumtwoPs}, we have 
\begin{equation} \label{thm3a}\frac32 -P_0(T) = \frac1{M}\sum_{j=1}^M P_{j-1/2}(T) + O\left(\frac{\sqrt{\log M}}{M}\right) + O\left(M\left(R(T)+R_P(T)+ \frac{1}{L^2}\right)\right) \end{equation} and 
\begin{equation} \label{thm3b}
P_0(T)-\frac12 = \frac1{M}\sum_{j=1}^M P_{j}(T) + O\left(\frac{\sqrt{\log M}}{M}\right) + O\left(M\left(R(T)+R_P(T)+ \frac{1}{L^2}\right)\right). \end{equation} 
\end{theorem}
If in Theorem \ref{thm3} we also assume \eqref{SumHalfk}, then by using either \eqref{SumHalfk} in \eqref{thm3a} or \eqref{Sumk} in \eqref{thm3b}, we immediately obtain
\begin{equation} \label{P_0eq} P_0(T) = 1 + O\left(\frac{\sqrt{\log M}}{M}\right)+ O\left(M\left(R(T)+R_P(T)+ \frac{1}{L^2}\right)\right). \end{equation}
Thus
\[ p_0 = \lim_{T\to\infty} P_0(T) = 1 + O\left(\frac{\sqrt{\log M}}{M}\right), \]
and since we can take $M$ as large as we wish, we have obtained $p_0=1$. Thus by Corollary \ref{cor1} we have obtained our main result on \hyperref[AH_GLSS]{\bf AH-Pairs} and \hyperref[AH_Weak]{\bf AH-Weak Density}.
\begin{theorem} \label{thm4} Assuming \hyperref[AH_GLSS]{\bf AH-Pairs} and \hyperref[AH_Weak]{\bf AH-Weak Density}, we have 
$p_0=1$ and asymptotically 100\% of the zeros of $\zeta(s)$ are simple and on the critical line.\end{theorem}

Without using \eqref{SumHalfk}, we cannot determine if the limiting densities in Theorem \ref{thm2} exist. However, we are still able to obtain a partial result on zeros without using \eqref{SumHalfk}.
\begin{corollary}\label{cor2}
Assuming \hyperref[AH_GLSS]{\bf AH-Pairs} and \eqref{SumtwoPs}, we have 
$\limsup_{T\to \infty}P_0(T)\le 3/2$ and asymptotically at least 50\% of the zeros of $\zeta(s)$ are simple and at least 50\% are on the critical line.
\end{corollary}

In Section \ref{sec2} we give the proof of Theorem \ref{p0thm}. In Section \ref{sec3} we present the main results from the method of Gallagher and Mueller we require in our proof of Theorem \ref{thm2}. In Section \ref{sec4} we give the proofs of Theorem \ref{thm2} and Corollary \ref{cor2}.

\section{Proof of Theorem \ref{p0thm}}
\label{sec2}

There is one technical detail which we first need to deal with. In \hyperref[AH_GLSS]{\bf AH-Pairs} and the densities defined using this conjecture, we consider zeros in the range $T/\log^2\!T< \gamma \le T$ rather than $0<\gamma \le T$. The reason for this is the Alternative Hypothesis requires we are in a range where the average spacing is $\sim 1/L$ as $T\to \infty$, and this is not true in the latter range. However, we want to use the usual range $0<\gamma \le T$ for counting zeros and defining \hyperref[PCC]{\bf PCC}, and therefore we need to move back and forth between the two ranges with an acceptable error. The estimate we use for this is, for $0\le h\le M$, and any continuous function $f(u)$ on the interval $|u|\le T$ with $0\le f(u) \le f(0)$,
\begin{equation}\label{switch}
\sum_{\substack{\rho,\rho' \\ (\gamma,\gamma')\in \mathcal{P}(T,M)\\ |(\gamma-\gamma')L|\le h} }f(\gamma-\gamma')
= \sum_{\substack{\rho,\rho' \\ 0 <\gamma,\gamma'\le T\\ |(\gamma-\gamma')L|\le h}} f(\gamma-\gamma') + O\left(f(0) \frac{(1+ h)T}{L}\right). \end{equation}

To obtain the error in \eqref{switch}, we first note that
$$ \sum_{\substack{\rho,\rho' \\ (\gamma,\gamma')\in \mathcal{P}(T,M)\\ |(\gamma-\gamma')L|\le h} }f(\gamma-\gamma')
= \sum_{\substack{\rho,\rho' \\ \frac{T}{\log^2\!T}<\gamma,\gamma'\le T \\ |(\gamma-\gamma')L|\le h}} f(\gamma-\gamma'), $$
then
\begin{align*}
\sum_{\substack{\rho,\rho' \\ 0 <\gamma,\gamma'\le T\\ |(\gamma-\gamma')L|\le h}}& f(\gamma-\gamma')
-
\sum_{\substack{\rho,\rho' \\ \frac{T}{\log^2\!T}<\gamma,\gamma'\le T \\ |(\gamma-\gamma')L|\le h}} f(\gamma-\gamma')\\&
=\sum_{\substack{\rho,\rho' \\ 0 <\gamma,\gamma'\le \frac{T}{\log^2\!T}\\ |(\gamma-\gamma')L|\le h}} f(\gamma-\gamma') +O\Bigg(\sum_{\substack{\rho,\rho' \\ \frac{T}{\log^2\!T} -\frac{h}{L} <\gamma,\gamma'\le \frac{T}{\log^2\!T}+\frac{h}{L}\\ |(\gamma-\gamma')L|\le h}} f(\gamma-\gamma')\Bigg)
\\
&\ll f(0)\sum_{\substack{\rho,\rho' \\ 0 <\gamma,\gamma'\le \frac{T}{\log^2\!T}+\frac{h}{L}\\ |(\gamma-\gamma')L|\le h}} 1 
\ll \ f(0)(1+h)\left(\frac{T}{\log^2\!T}+\frac{h}{L}\right)L\ll f(0) \frac{(1+ h)T}{L} ,
\end{align*}
where we used \eqref{GMbound} in the last line and $h/L =o(1)$ since $h\le M$, where $M$ is fixed as $T\to \infty$.

\begin{proof}[Proof of Theorem \ref{p0thm}]
Assuming \hyperref[AH_GLSS]{\bf AH-Pairs}, we have for sufficiently large $T$ that there exists a constant $C_0>0$ with
\begin{equation}\label{eq:cor1:eq1}
\begin{aligned}
    P_0(T)\, TL = |\mathcal{B}_0(T)| 
    &= 
    \sum_{\substack{\rho,\rho' \\ \frac{T}{\log^2\!T}<\gamma,\gamma'\le T \\ -\frac{1}{4} < (\gamma-\gamma')L \le \frac{1}{4}}} 1
    =
    \sum_{\substack{\rho,\rho' \\ \frac{T}{\log^2\!T}<\gamma,\gamma'\le T \\ |(\gamma-\gamma')L| \le C_0R(T)}} 1 \\
    &=
    \sum_{\substack{\rho,\rho' \\ 0<\gamma,\gamma'\le T \\ |(\gamma-\gamma')L| \le C_0R(T)}} 1 + O\left(\frac{(1 + R(T))T}{L}\right) \\
    &=
    N^\circledast(T) + 2N(T,C_0R(T)) + O\left(\frac{T}{L}\right) ,
\end{aligned}
\end{equation}
where we used \eqref{switch}.

If \hyperref[ESH]{\bf ESH} or equivalently \eqref{ES} holds,
then $N^\circledast(T) + 2N(T,C_0R(T)) = TL+o(TL)$ as $T \to \infty$,
hence by \eqref{eq:cor1:eq1},
we have $P_0(T)\, TL = TL + o(TL)$ and thus $p_0=1$.

Conversely, 
if $p_0=1$,
then the left-hand side of Equation \eqref{eq:cor1:eq1} is
$=TL + o(TL)$
so we obtain
\[
    N(T) 
    \le 
    N^\circledast(T) + 2N(T,\lambda_0) 
    = \sum_{\substack{\rho,\rho' \\ 0<\gamma, \gamma'\le T \\ |(\gamma-\gamma')L|\le \lambda_0 }} 1 
    \le
    P_0(T)\, TL +O\left( \frac{T}{L}\right)
    =
    TL + o(TL),
    \quad
    \text{as}~ T \to \infty,
\]
for any $\lambda_0< 1/4$
where the first inequality follows from Equation \eqref{N_cir_star_lower}.
Thus we obtain \eqref{ES}, which
is equivalent to \hyperref[ESH]{\bf ESH}.
\end{proof}

\section{The Method of Gallagher and Mueller}
\label{sec3}

We use the same method as \cite{GLSS1}, following Gallagher and Mueller \cite{GaMu78}. Both propositions below come from \cite[Section 2]{GLSS1}, and the proofs are given in Section 4 of \cite{GLSS1}.
The key idea lies in the estimate of the summation over pairs of zeros
\begin{equation}\label{Dsec2a} \Ss(T,\lambda) := \sum_{\substack{\rho,\rho' \\ 0<\gamma, \gamma'\le T \\ |(\gamma-\gamma')L|\le \lambda }} \left(\frac\lambda{L}-|\gamma-\gamma'|\right),
\end{equation}
for a real number $\lambda >0$, which is closely related to the second moment for zeros in short intervals.
\begin{prop}[Gallagher and Mueller {\cite{GaMu78}}] \label{prop1} For $\lambda>0$, we have, as $T\to\infty$,
\[ \int_0^T \left(N\left(t+\tfrac\lambda{L}\right)-N(t)\right)^2 \, dt
= \Ss(T,\lambda) + O(L^2). \]
\end{prop}
On the other hand, we also have the following unconditional estimate for the above second moment of zeros. 
Recall for $T\ge 2$, the Riemann-von Mangoldt formula
\[
S(T)
:=
\frac1{\pi}{\rm arg }\, \zeta(\tfrac12+iT)
=
N(T)
-
\frac{T}{2\pi} \log \frac{T}{2\pi e}
-
\frac78
+ O\left(\frac{1}{T}\right),
\]
where the term $O(T^{-1})$ is continuous, see Titchmarsh \cite[Theorem 9.3]{Titchmarsh2} or \cite[Corollary 14.2]{MontgomeryVaughan2007}. Here we assume $T\neq \gamma$ for any zero $\rho=\beta+i\gamma$, and we set $N(T)= N(T^+)$ otherwise.
\begin{prop}[Gallagher and Mueller {\cite{GaMu78}}, Fujii {\cite{Fu74,Fu81}}, Tsang {\cite{Tsang84}}]\label{prop2} For $\lambda>0$, we have, as $T\to\infty$,
\begin{equation*}
\int_0^T \left(N\left(t+\tfrac\lambda{L}\right)-N(t)\right)^2 \, dt = \lambda^2T + O\left(\frac{\lambda^2T}L\right) + \int_0^T \left(S\left(t+\tfrac\lambda{L}\right)-S(t)\right)^2 \, dt + O(L^2),
\end{equation*}
and
\begin{equation*}
\int_0^T \left(S\left(t+\tfrac\lambda{L}\right) - S(t)\right)^2\, dt = \frac{T}{\pi^2} \log(2+\lambda) + O\left(T\sqrt{\log(2+\lambda)}\right).
\end{equation*}
\end{prop}

Combining the two results, we have for any $\lambda>0$, as $T\to\infty$,
\begin{equation} \begin{split} \label{Dsec2b}
\Ss(T,\lambda)-\lambda^2T &= \int_0^T \left(S\left(t+\tfrac\lambda{L}\right) - S(t)\right)^2\, dt + O\left(\frac{\lambda^2T}L\right) + O(L^2) \\
&= \frac{T}{\pi^2} \log(2+\lambda) + O\left(T\sqrt{\log(2+\lambda)}\right) + O\left(\lambda^2T/L\right).
\end{split}\end{equation}

\section{Proof of Theorem \ref{thm2} and Corollary \ref{cor2}}
\label{sec4}

\begin{proof}[Proof of Theorem \ref{thm2}]
We first evaluate $\Ss(T,M)$ using \hyperref[AH_GLSS]{\bf AH-Pairs}.
Letting $M$ be a positive even integer and taking $\lambda=M$ in \eqref{Dsec2a}, we have on using \eqref{switch} and then \hyperref[AH_GLSS]{\bf AH-Pairs} that 
\[ \begin{split} \Ss(T,M) &= \sum_{\substack{\rho,\rho' \\ 0<\gamma, \gamma'\le T \\ |(\gamma-\gamma')L|\le M }} \left(\frac{M}{L}-|\gamma-\gamma'|\right) = \sum_{\substack{\rho,\rho' \\ (\gamma,\gamma')\in \mathcal{P}(T,M) \\ |(\gamma-\gamma')L|\le M }} \left(\frac{M}{L}-|\gamma-\gamma'|\right) +O\left( \frac{M}{L}\frac{(1+M)T}{L}\right) \\&
= \sum_{|k|\le 2M} \sum_{\substack{\rho, \rho' \\ (\gamma, \gamma') \in \mathcal{B}_{k/2}(T)}} \left(\frac{M}{L}-|\gamma-\gamma'|\right)+O\left( \frac{M^2T}{L^2}\right)\\
&= T\sum_{|k|\le 2M} \bigl(M-{|k|}/{2} + O\left((|k|+1) R(T)\right)\bigr) P_{k/2}(T) + O\left( \frac{M^2T}{L^2}\right).
\end{split}\]

Writing $k=2j$ if $k$ is even, and $k=2j-1$ otherwise,  we have by \eqref{P_k/2sumbound} as $T\to\infty$,
\begin{equation} \begin{split}\label{D(T,M)1}
\Ss(T,M) &= T\big(M+O(R(T)\big)P_0(T)+ 2T\sum_{k=1}^{2M}\left(M- k/2 + O\left(k R(T)\right)\right) P_{k/2}(T)  +O\left( \frac{M^2T}{L^2}\right) \\
&= \left(MP_0(T) + 2\sum_{k=1}^{2M} \left(M-k/2\right) P_{k/2}(T) + O\left(MR(T)\sum_{k=0}^{2M}P_{k/2}(T) +\frac{M^2}{L^2}\right) \right)T \\
&= \Biggl(MP_0(T) + 2\sum_{j=1}^{M}\left( (M-j+1/2)P_{j-1/2} + (M-j)P_j(T) \right) \\
&\qquad+ O\left(M^2\left(R(T) + \frac1{L^2}\right)\right)\Biggr)T \\
&= \Biggl(MP_0(T) + 2\sum_{j=1}^{M} (M-j)(P_{j-1/2}(T) + P_j(T)) + \sum_{j=1}^{M}P_{j-1/2}(T) \\
&\qquad+ O\left(M^2\left(R(T)+ \frac1{L^2}\right)\right)\Biggr)T.
\end{split}\end{equation}

Next, taking $\lambda=M$ in \eqref{Dsec2b}, we have as $T\to \infty$
\begin{equation}\label{D(T,M)2}
\begin{aligned}
\Ss(T,M) &= \left(M^2 + \frac{\log M}{\pi^2} + O\left(\sqrt{\log M}\right)\right)T \\
&= \left(\frac32 M + 2\sum_{j=1}^{M}(M-j)\left(1-\frac{2}{\pi^2(2j-1)^2} \right)  + O\left(\sqrt{\log M}\right) \right)T,
\end{aligned}
\end{equation}
where the second equality is obtained by noting that
\begin{equation*}
\begin{aligned}
2\sum_{j=1}^{M} (M-j)\left(1-\frac{2}{\pi^2(2j-1)^2} \right)
&= 2\sum_{j=1}^{M}(M-j) - \frac{4}{\pi^2}\sum_{j=1}^{M}\frac{M-j}{(2j-1)^2} \\
&=
\left( 
    M^2 - M \right) - \left(
    \frac{M}2 - \frac{\log M}{\pi^2} + O(1)
\right) \\
&= M^2 -\frac32M + \frac{\log M}{\pi^2} + O(1).
\end{aligned}
\end{equation*}
Combining \eqref{D(T,M)1} and \eqref{D(T,M)2} and dividing by $T$, we obtain \eqref{mainthm}.
\end{proof}

\begin{proof}[Proof of Corollary \ref{cor2}] Assuming \hyperref[AH_GLSS]{\bf AH-Pairs} and \eqref{SumtwoPs}, by \eqref{thm3a} and  $P_{k/2}(T)\ge 0$ we have 
\[ P_0(T)\le P_0(T) +\frac1{M}\sum_{j=1}^M P_{j-1/2}(T) = \frac32 + O\left(\frac{\sqrt{\log M}}{M}\right) + O\left(M\left(R(T)+R_P(T)+ \frac{1}{L^2}\right)\right).\]
Thus, as $T\to \infty$,
\begin{equation}
\limsup_{T\to \infty}P_0(T) \le \frac32 + O\left(\frac{\sqrt{\log M}}{M}\right),
\qquad \text{and thus}
\qquad
P_0(T) \le \frac32 + o(1)
\end{equation}
because $M$ can be taken as large as we wish.
Next, by \eqref{eq:cor1:eq1} we have as $T\to \infty$
\begin{align*}
N^\circledast(T)\le N^\circledast(T) + 2N(T,C_0R(T))
&= \sum_{\substack{\rho,\rho' \\ 0<\gamma, \gamma'\le T \\ |(\gamma-\gamma')L|\le C_0R(T) }} 1 =  \sum_{\substack{\rho,\rho' \\ \frac{T}{\log^2\!T}<\gamma,\gamma'\le T \\ |(\gamma-\gamma')L| \le C_0R(T)}} 1 + O\left(\frac{T}{L}\right)\\
&= |\mathcal{B}_0(T)| +O\left(\frac{T}{L}\right)
= P_0(T)\, TL  + o(TL)
\le \left(\frac32+o(1)\right) TL.
\end{align*}
From \cite[Equation (1.10)]{GLSS1}
we have
\begin{equation}\label{cor2proof} N^*(T) \le
N^*(T) + N_{\beta\neq \frac12}^*(T) \le N^\circledast(T) \le 
\left(\frac32+o(1)\right) TL,
\end{equation}
where 
\[ N_{\beta\neq \frac12}^*(T) := \sum_{\substack{\rho\\ 0<\gamma \le T\\ \beta \neq \frac12}}m_\rho . \]
Thus, by the same argument as in \cite[Equation (1.11) and below]{GLSS1}, we have, letting
\begin{equation}\label{Thm1eq} N_0(T) := \sum_{\substack{\rho\\0<\gamma \le T\\ \beta=\frac12}} 1 \qquad \text{and}\qquad N_s(T) := \sum_{\substack{\rho\ \text{simple} \\ 0<\gamma\le T}} 1,\end{equation} that, by \eqref{cor2proof},
\[ N_s(T) \ge \sum_{\substack{\rho \\ 0<\gamma\le T}} (2-m_\rho)
= 2N(T)- N^*(T) \ge 2N(T) - \left(\frac32+o(1)\right)TL = \left(\frac12+o(1)\right)N(T). \]
Similarly, since $N^*(T) \ge N(T)$, we have by \eqref{cor2proof}
\[
 N_{\beta\neq \frac12}^*(T)\le 
\left(\frac32+o(1)\right)TL
- N^*(T)
\le 
\left(\frac12+o(1)\right)N(T) , \]
which implies since $N_0(T) + N_{\beta\neq \frac12}^*(T)\ge N(T)$ that 
$N_0(T) \ge (\frac12+o(1)) N(T)$.
\end{proof}

\section*{Acknowledgement of Funding}

The first, second, and fourth authors thank the American Institute of Mathematics for its hospitality and for providing a pleasant research environment where this research project first started. The second author was supported by the National Research Foundation of Korea (NRF) grant funded by the government of the Republic of Korea (MSIT) (No. RS-2025-00553763, RS-2024-00415601, and RS-2024-00341372). The fourth author was supported by the Inamori Research Grant 2024 and JSPS KAKENHI Grant Number 22K13895.

\bibliographystyle{alpha}
\bibliography{AHReferences}

\end{document}